\documentclass[10pt]{amsart}
\usepackage{amssymb,amsmath,amsfonts,latexsym,amsthm,geometry}
\usepackage{amsmath}
\usepackage{amsfonts}
\usepackage{amssymb}
\usepackage{graphicx}
\usepackage{color, soul}

\providecommand{\U}[1]{\protect\rule{.1in}{.1in}}
\providecommand{\U}[1]{\protect\rule{.1in}{.1in}}
\providecommand{\U}[1]{\protect\rule{.1in}{.1in}}
\newtheorem{theorem}{Theorem}[section]
\newtheorem{corollary}[theorem]{Corollary}
\newtheorem{proposition}[theorem]{Proposition}

\theoremstyle{definition}

\begin{document}
\title[On the mixed $\left( \ell _{1},\ell _{2}\right) $-Littlewood
inequalities and interpolation ]{On the mixed $\left( \ell _{1},\ell
_{2}\right) $-Littlewood inequalities and interpolation}
\author[M. Maia and J. Santos]{Mariana Maia and Joedson Santos}
\address{Departamento de Matem\'{a}tica, Universidade Federal da Para\'{\i}%
ba, 58.051-900 - Jo\~{a}o Pessoa, Brazil.}
\email{joedsonmat@gmail.com and mariana.britomaia@gmail.com}
\thanks{Joedson Santos is supported by CNPq Grant 303122/2015-3}

\begin{abstract}
It is well-known that the optimal constant of the bilinear
Bohnenblust--Hille inequality (i.e., Littlewood's $4/3$ inequality) is
obtained by interpolating the bilinear mixed $\left( \ell _{1},\ell
_{2}\right) $-Littlewood inequalities. We remark that this cannot be
extended to the $3$-linear case and, in the opposite direction, we show that
the asymptotic growth of the constants of the $m$-linear Bohnenblust--Hille
inequality is the same of the constants of the mixed $\left( \ell _{\frac{%
2m+2}{m+2}},\ell _{2}\right) $-Littlewood inequality. This means that,
contrary to what the previous works seem to suggest, interpolation does not
play a crucial role in the search of the exact asymptotic growth of the
constants of the Bohnenblust--Hille inequality. In the final section we use
mixed Littlewood type inequalities to obtain the optimal cotype constants of
certain sequence spaces.
\end{abstract}

\maketitle


\section{Introduction}

The mixed $\left( \ell _{1},\ell _{2}\right) $-Littlewood inequality for
$\mathbb{K}=\mathbb{R}$ or $\mathbb{C}$ asserts that
\begin{equation}
\sum_{j_{1}=1}^{\infty }\left( \sum_{j_{2},...,j_{m}=1}^{\infty }\left\vert
U(e_{j_{1}},...,e_{j_{m}})\right\vert ^{2}\right) ^{\frac{1}{2}}\leq \left(
\sqrt{2}\right) ^{m-1}\left\Vert U\right\Vert ,  \label{u8}
\end{equation}%
for all continuous $m$-linear forms $U:c_{0}\times \cdots \times
c_{0}\rightarrow \mathbb{K}$, where $\left( e_{i}\right) _{i=1}^{\infty }$
denotes the sequence of canonical vectors of $c_{0}$. It is well-known that
arguments of symmetry combined with an inequality due to Minkowski yields
that for each $k\in \{2,...,m\}$ we have%
\begin{equation}
\left( \sum_{j_{1},...,j_{k-1}=1}^{\infty }\left( \sum_{j_{k}=1}^{\infty
}\left( \sum_{j_{k+1},...,j_{m}=1}^{\infty }\left\vert
U(e_{j_{1}},...,e_{j_{m}})\right\vert ^{2}\right) ^{\frac{1}{2}\times
1}\right) ^{\frac{1}{1}\times 2}\right) ^{\frac{1}{2}}\leq \left( \sqrt{2}%
\right) ^{m-1}\left\Vert U\right\Vert ,  \label{0009}
\end{equation}%
which is also called mixed $\left( \ell _{1},\ell _{2}\right) $-Littlewood
inequality. \ For the sake of simplicity we can say that we have $m$
inequalities with \textquotedblleft multiple\textquotedblright\ exponents $%
\left( 1,2,2,...,2\right) ,...,\left( 2,...,2,1\right) $. These inequalities
are in the heart of the proof of the famous Bohnenblust--Hille inequality
for multilinear forms (\cite{bh}) which states that there exists a sequence
of positive scalars $\left( B_{m}^{\mathbb{K}}\right) _{m=1}^{\infty }$ in $%
[1,\infty )$ such that
\begin{equation}
\left( \sum\limits_{i_{1},\ldots ,i_{m}=1}^{\infty }\left\vert
U(e_{i_{^{1}}},\ldots ,e_{i_{m}})\right\vert ^{\frac{2m}{m+1}}\right) ^{%
\frac{m+1}{2m}}\leq B_{m}^{\mathbb{K}}\left\Vert U\right\Vert   \label{ul}
\end{equation}%
for all continuous $m$-linear forms $U:c_{0}\times \cdots \times
c_{0}\rightarrow \mathbb{K}$. This inequality is essentially a
result of the successful theory of nonlinear absolutely summing
operators (for more details on  summing operators see, for instance,
\cite{blasco, popa, rueda} and references therein). To prove the
Bohnenblust--Hille inequality using the mixed $\left( \ell _{1},\ell
_{2}\right) $-Littlewood inequalities it suffices to observe that
the exponent $\frac{2m}{m+1}$ can be seen as a multiple exponent
$\left( \frac{2m}{m+1},...,\frac{2m}{m+1}\right) $ and this exponent
is precisely the interpolation of the exponents $\left(
1,2,2,...,2\right) ,...,\left( 2,...,2,1\right) $ with weights
$\theta _{1}=\cdots =\theta _{m}=1/m$. Mixed Littlewood inequalities
are also crucial to prove Hardy--Littlewood inequalities for
multilinear forms (see \cite{ara, LLL} and the references therein).

\section{Mixed Littlewood inequalities and interpolation}

The optimal constant of the $3$-linear mixed $\left( \ell _{1},\ell
_{2}\right)$-Littlewood inequality for real scalars with multiple exponents $\left(
1,2,2\right) $ and $\left( 2,1,2\right) $ were obtained in \cite{natal,
diana} (these constants are precisely $2$). Curiously, the arguments could
not be extended to obtain the optimal constant associated to the multiple
exponent $\left( 2,2,1\right) .$ However, using the $3$-linear form%
\begin{equation*}
U(x,y,z)=\left( z_{1}+z_{2}\right) \left(
x_{1}y_{1}+x_{1}y_{2}+x_{2}y_{1}-x_{2}y_{2}\right) +\left(
z_{1}-z_{2}\right) \left( x_{3}y_{3}+x_{3}y_{4}+x_{4}y_{3}-x_{4}y_{4}\right)
\end{equation*}%
it is easy to show that the optimal constant associated to the multiple
exponent $\left( 2,2,1\right) $ is not smaller than $\sqrt{2}.$ So,
interpolating the three inequalities we obtain the estimate $2^{1/3}\times
2^{1/3}\times \sqrt{2}^{1/3}$ for the $3$-linear Bohnenblust--Hille
inequality, i.e., $2^{5/6}$, but it is well-known that the optimal constant
of the $3$-linear Bohnenblust--Hille inequality is not bigger than $2^{3/4}.$
So we conclude that the optimal constant of the $3$-linear
Bohnenblust--Hille inequality cannot be obtained by interpolating the
optimal constants of the multiple exponents $\left( 1,2,2\right) ,$ $\left(
2,1,2\right) $ and $\left( 2,2,1\right) .$

In the paper \cite{aa}, Albuquerque \textit{et al.} have shown that the
Bohnenblust--Hille inequality is a very particular case of the following
theorem:

\begin{theorem}
\label{THMBHQ} Let $1\leq k\leq m$ and $n_{1},\ldots ,n_{k}\geq 1$ be
positive integers such that $n_{1}+\cdots +n_{k}=m$, let $q_{1},\dots
,q_{k}\in \lbrack 1,2]$. The following assertions are equivalent:

(A) There is a constant $C_{k,q_{1}...q_{k}}^{\mathbb{K}}\geq 1$ such that%
{\small {\
\begin{equation}
\left( {\sum\limits_{i_{1}=1}^{\infty }}\left( {\sum\limits_{i_{2}=1}^{%
\infty }}\left( ...\left( {\sum\limits_{i_{k-1}=1}^{\infty }}\left( {%
\sum\limits_{i_{k}=1}^{\infty }}\left\vert A\left( e_{i_{1}}^{n_{1}},\ldots
,e_{i_{k}}^{n_{k}}\right) \right\vert ^{q_{k}}\right) ^{\frac{q_{k-1}}{q_{k}}%
}\right) ^{\frac{q_{k-2}}{q_{k-1}}}\cdots \right) ^{\frac{q_{2}}{q_{3}}%
}\right) ^{\frac{q_{1}}{q_{2}}}\right) ^{\frac{1}{q_{1}}}\leq
C_{k,q_{1}...q_{k}}^{\mathbb{K}}\left\Vert A\right\Vert  \label{778}
\end{equation}%
}}for all continuous $m$-linear forms $A:c_{0}\times \cdots \times
c_{0}\rightarrow \mathbb{K}$.

(B) $\frac{1}{q_{1}}+\cdots+\frac{1}{q_{k}}\leq\frac{k+1}{2}.$
\end{theorem}

The inequalities (\ref{778}) when $k=m$, $q_{j}=2$ and
$q_{l}=\frac{2m-2}{m}$ for all $l\in\{1,...,j-1,j+1,...,m\}$ can be called mixed $\left( \ell _{\frac{%
2m-2}{m}},\ell _{2}\right) $-Littlewood inequality for short (see \cite%
{natal}). The best constants $C_{\frac{2m}{m+1}...\frac{2m}{m+1}}^{\mathbb{K}%
}$ ($C_{m}^{\mathbb{K}}$ for short) are unknown (even its asymptotic growth
is unknown). We stress that it is even unknown if the sequence $\left(
C_{m}^{\mathbb{K}}\right) _{m=1}^{\infty }$ is increasing. By the Khinchin
inequality it can be proved (see \cite{adv}) that%
\begin{equation}
C_{2,\frac{2m-2}{m},...,\frac{2m-2}{m}}^{\mathbb{K}}\leq A_{\frac{2m-2}{m}%
}^{-1}C_{m-1}^{\mathbb{K}}.  \label{33}
\end{equation}%
where $A_{p}$ are the optimal constants of the Khinchin inequality. Using an
interpolative procedure, or the H\"{o}lder inequality for mixed sums, this
means that%
\begin{equation*}
C_{m}^{\mathbb{K}}\leq A_{\frac{2m-2}{m}}^{-1}C_{m-1}^{\mathbb{K}}.
\end{equation*}

We shall prove the following asymptotic equivalences:%
\begin{equation}
C_{m-1}^{\mathbb{K}}\sim C_{2,\frac{2m-2}{m},...,\frac{2m-2}{m}}^{\mathbb{K}%
}\sim \cdots \sim C_{\frac{2m-2}{m},...,\frac{2m-2}{m},2}^{\mathbb{K}}
\label{800}
\end{equation}%
that seem to have been overlooked until now. This means that the search of
the precise asymptotic growth of the best constants of the
Bohnenblust--Hille inequality is equivalent to the search of the precise
asymptotic growth of, for instance, the sequence $\left( C_{2,\frac{2m-2}{m}%
,...,\frac{2m-2}{m}}^{\mathbb{K}}\right) _{m=1}^{\infty }$ and no
interpolative procedure is needed. As a corollary conclude that the
inequality (\ref{33}) is asymptotically sharp.

The proof of (\ref{800}) is simple. If $T_{m-1}$ is a $\left( m-1\right) $%
-linear form, we define
\begin{equation*}
T_{m}(x^{(1)},...,x^{(m)})=T_{m-1}(x^{(2)},...,x^{(m)})x_{1}^{(1)}.
\end{equation*}%
Then%
\begin{eqnarray*}
&&\left( \sum_{j_{2},...,j_{m}=1}^{\infty }\left\vert T_{m-1}\left(
e_{j_{2},...,}e_{j_{m}}\right) \right\vert ^{\frac{2m-2}{m}}\right) ^{\frac{m%
}{2m-2}} \\
&=&\left( \sum_{j_{1}=1}^{\infty }\left( \sum_{j_{2},...,j_{m}=1}^{\infty
}\left\vert T_{m}\left( e_{j_{1},...,}e_{j_{m}}\right) \right\vert ^{\frac{%
2m-2}{m}}\right) ^{\frac{m}{2m-2}2}\right) ^{\frac{1}{2}} \\
&\leq &C_{2,\frac{2m-2}{m},...,\frac{2m-2}{m}}^{\mathbb{K}}\left\Vert
T_{m}\right\Vert \\
&=&C_{2,\frac{2m-2}{m},...,\frac{2m-2}{m}}^{\mathbb{K}}\left\Vert
T_{m-1}\right\Vert .
\end{eqnarray*}%
We thus conclude that%
\begin{equation*}
C_{m-1}^{\mathbb{K}}\leq C_{2,\frac{2m-2}{m},...,\frac{2m-2}{m}}^{\mathbb{K}%
}.
\end{equation*}%
Therefore%
\begin{equation*}
C_{m-1}^{\mathbb{K}}\leq C_{2,\frac{2m-2}{m},...,\frac{2m-2}{m}}^{\mathbb{K}%
}\leq A_{\frac{2m-2}{m}}^{-1}C_{m-1}^{\mathbb{K}}.
\end{equation*}%
Since (for both real and complex scalars)%
\begin{equation*}
\lim_{m\rightarrow \infty }A_{\frac{2m-2}{m}}^{-1}=1,
\end{equation*}%
we conclude that
\begin{equation*}
C_{m-1}^{\mathbb{K}}\sim C_{2,\frac{2m-2}{m},...,\frac{2m-2}{m}}^{\mathbb{K}%
}.
\end{equation*}%
The other equivalences are similar.

\section{Cotype $2$ constants of $\ell _{p}$ spaces}

Let $2\leq q<\infty $ and $0<s<\infty $. A Banach space $X$ has cotype $q$
(see \cite[page 138]{albiac}) if there is a constant $C_{q,s}>0$ such that,
no matter how we select finitely many vectors $x_{1},\dots ,x_{n}\in X$,%
\begin{equation}
\left( \sum_{k=1}^{n}\Vert x_{k}\Vert ^{q}\right) ^{\frac{1}{q}}\leq
C_{q,s}\left( \int_{[0,1]}\left\Vert \sum_{k=1}^{n}r_{k}(t)x_{k}\right\Vert
^{s}dt\right) ^{1/s},  \label{99}
\end{equation}%
where $r_{k}$ denotes the $k$-th Rademacher function. The smallest of all of
these constants will be denoted by $C_{q,s}(X).$

By the Kahane inequality we know that if (\ref{99}) holds for a certain $s>0$
than it holds for all $s>0.$ It is well-known that for all $p\geq 1$, the
sequence space $\ell _{p}$ has cotype $\max \{p,2\}.$ The optimal values of $%
C_{2,s}(\ell _{p})$ for $1\leq p<2$ are perhaps known or at least folklore,
but we were not able to find in the literature. Classical books like \cite%
{albiac, diestel, garling} do not provide this information.

In this section we shall show how the optimal cotype constant of $\ell _{p}$
spaces can be obtained using mixed inequalities similar to those treated in
the previous section. From now on, $p_{0}$ is the solution of the following
equality%
\begin{equation*}
\Gamma \left( \frac{p_{0}+1}{2}\right) =\frac{\sqrt{\pi }}{2}.
\end{equation*}

\begin{theorem}
Let $1\leq r\leq p_{0}\approx 1.84742.$ Then%
\begin{equation*}
C_{2,r}(\ell _{r})=2^{\frac{1}{r}-\frac{1}{2}}.
\end{equation*}
\end{theorem}

\begin{proof}
It is not difficult to prove that $C_{2,r}(\ell _{r})\leq 2^{\frac{1}{r}-%
\frac{1}{2}}$ (see \cite[pages 141-142]{albiac}). Now we prove that $2^{%
\frac{1}{r}-\frac{1}{2}}$ is the best constant possible.

Let $A:c_{0}\times c_{0}\rightarrow \mathbb{R}$ a bilinear form and define,
for all positive integers $n$,%
\begin{equation*}
A_{n,e}:c_{0}\rightarrow \ell _{r}
\end{equation*}%
by%
\begin{equation*}
A_{n,e}(x)=\left( A\left( x,e_{k}\right) \right) _{k=1}^{n}.
\end{equation*}%
It is simple to verify that%
\begin{equation*}
\left\Vert A_{n,e}\right\Vert \leq \left\Vert A\right\Vert .
\end{equation*}

In fact,%
\begin{eqnarray*}
\left\Vert A_{n,e}\right\Vert &=&\sup_{\left\Vert x\right\Vert \leq
1}\left\Vert A_{n,e}\left( x\right) \right\Vert =\sup_{\left\Vert
x\right\Vert \leq 1}\left( \sum\limits_{j=1}^{n}\left\vert A\left(
x,e_{j}\right) \right\vert ^{r}\right) ^{1/r} \\
&\leq &\sup_{\left\Vert x\right\Vert \leq 1}\pi _{(r,r)}\left( A\left(
x,\cdot \right) \right) \sup_{\varphi \in B_{\left( c_{0}\right) ^{\ast
}}}\left( \sum\limits_{j=1}^{n}\left\vert \varphi \left( e_{j}\right)
\right\vert ^{r}\right) ^{1/r} \\
&\leq &\sup_{\left\Vert x\right\Vert \leq 1}\left\Vert A\left( x,\cdot
\right) \right\Vert \sup_{\varphi \in B_{\left( c_{0}\right) ^{\ast
}}}\sum\limits_{j=1}^{n}\left\vert \varphi \left( e_{j}\right) \right\vert \\
&=&\left\Vert A\right\Vert .
\end{eqnarray*}

It is also well-known that $A_{n,e}$ is absolutely $\left( 2,1\right) $%
-summing and%
\begin{equation*}
\pi _{(2,1)}\left( A_{n,e}\right) \leq C_{2,r}(\ell _{r})\left\Vert
A_{n,e}\right\Vert .
\end{equation*}%
In fact, for any continuous linear operator $u:c_{0}\rightarrow \ell _{r}$
we have%
\begin{eqnarray*}
\left( \sum_{j=1}^{n}\left\Vert u\left( x_{j}\right) \right\Vert ^{2}\right)
^{\frac{1}{2}} &\leq &C_{2,r}(\ell _{r})\left( \int_{[0,1]}\left\Vert
\sum_{j=1}^{n}r_{j}(t)u\left( x_{j}\right) \right\Vert ^{r}dt\right) ^{\frac{%
1}{r}} \\
&\leq &C_{2,r}(\ell _{r})\sup_{t\in \lbrack 0,1]}\left\Vert
\sum_{j=1}^{n}r_{j}(t)u\left( x_{j}\right) \right\Vert \\
&= &C_{2,r}(\ell _{r})\left\Vert u\right\Vert \sup_{\varphi \in B_{\left(
c_{0}\right) ^{\ast }}}\sum\limits_{j=1}^{n}\left\vert \varphi \left(
x_{j}\right) \right\vert.
\end{eqnarray*}

We have%
\begin{eqnarray}
\left( \sum_{j_{1}=1}^{n}\left( \sum_{j_{2}=1}^{n}\left\vert
A(e_{j_{1}},e_{j_{2}})\right\vert ^{r}\right) ^{\frac{1}{r}\times 2}\right)
^{\frac{1}{2}} &=&\left( \sum_{j_{1}=1}^{n}\left\Vert A_{n,e}\left(
e_{j_{1}}\right) \right\Vert ^{2}\right) ^{\frac{1}{2}}  \label{11} \\
&\leq &C_{2,r}(\ell _{r})\left\Vert A_{n,e}\right\Vert \sup_{\varphi \in
B_{\left( c_{0}\right) ^{\ast }}}\sum\limits_{j=1}^{n}\left\vert \varphi
\left( e_{j}\right) \right\vert  \notag \\
&\leq &C_{2,r}(\ell _{r})\left\Vert A\right\Vert .  \notag
\end{eqnarray}%
But, plugging%
\begin{equation*}
A(x,y)=x_{1}y_{1}+x_{1}y_{2}+x_{2}y_{1}-x_{2}y_{2}
\end{equation*}%
into (\ref{11}) we conclude that%
\begin{equation*}
\left( 2\cdot 2^{\frac{2}{r}}\right) ^{\frac{1}{2}}\leq 2C_{2,r}(\ell _{r})
\end{equation*}%
and thus%
\begin{equation*}
C_{2,r}(\ell _{r})\geq \frac{2^{\frac{1}{2}+\frac{1}{r}}}{2}=2^{\frac{1}{r}-%
\frac{1}{2}}.
\end{equation*}
\end{proof}

\bigskip A simple adaptation of the above proof gives us:

\begin{proposition}
Let $1\leq r\leq 2.$ Then%
\begin{equation*}
C_{2,s}(\ell _{r})\geq 2^{\frac{1}{r}-\frac{1}{2}}
\end{equation*}%
for all $s>0.$
\end{proposition}

\bigskip

The same argument of the previous result provides:

\begin{corollary}
\bigskip Let $p_{0}\approx 1.84742<r\leq 2.$ Then%
\begin{equation*}
2^{\frac{1}{r}-\frac{1}{2}}\leq C_{2,r}(\ell _{r})\leq \frac{1}{\sqrt{2}}%
\left( \frac{\Gamma (\frac{r+1}{2})}{\sqrt{\pi }}\right) ^{-1/r}.
\end{equation*}
\end{corollary}

\end{document}